\documentclass{amsart}
\usepackage{verbatim}
\usepackage{rotating} %for \includegraphix[angle=270]
\usepackage{amssymb}
\usepackage{mathrsfs,amsfonts,amsmath,amssymb,epsfig,amscd,xy,amsthm,pb-diagram} 
\newtheorem{theorem}{Theorem}[section]
\newtheorem{proposition}[theorem]{Proposition}

\newtheorem{lemma}[theorem]{Lemma}

\newtheorem{corollary}[theorem]{Corollary}

\newtheorem{definition}[theorem]{Definition}
\newtheorem{Problems}[theorem]{Problems}

\theoremstyle{plain}
\theoremstyle{remark}

\newtheorem{remark}[theorem]{Remark}
\newtheorem{example}[theorem]{Example}

\newcommand{\C}{{\mathbb C}}

\newcommand{\Q}{{\mathbb Q}}

\newcommand{\R}{{\mathbb R}}

\newcommand{\Z}{{\mathbb Z}}
\newcommand{\N}{{\mathbb N}}

\newcommand{\Qbar}{\bar{\Q}}

\DeclareMathOperator{\Gal}{Gal}

\DeclareMathOperator{\Norm}{N}

\DeclareMathOperator{\Supp}{Supp}

\newcommand{\scrI}{\mathscr{I}}

\author{Khoa D.~Nguyen}
\address{
Khoa D.~Nguyen \\
Department of Mathematics and Statistics\\
University of Calgary\\
2500 University Drive NW\\
Calgary, T2N 1N4, Alberta, Canada
}
\email{dangkhoa.nguyen@ucalgary.ca}

\thanks{We are partially supported by the NSERC Discovery Grant RGPIN-2025-04439 and a Canada Research Chair Tier-2 funding from the Government of Canada.}

\begin{document}
	\title[Salem's third problem]{Algebraic numbers and Fourier analysis: Salem's third problem}
	
	\date{April 2026}
	
	\begin{abstract}
 	In 1963, Rapha{\"e}l Salem concluded his highly influential book ``Algebraic Numbers and Fourier Analysis'' with a list of four unsolved problems. The first two problems remain wide open while the last problem on the absolute continuity of Bernoulli convolutions has seen significant progress over the years including recent results by Shmerkin and Varj\'u. In this paper, we solve the third problem concerning the vanishing at infinity of the
 	product of Fourier transforms of Bernoulli convolutions each of which does not vanish at infinity.  Our solution uses tools in diophantine approximation such as the theory of Weil heights and Lang's general formulation of Roth's theorem.
	\end{abstract}
	
	\maketitle

	\section{Introduction}\label{sec:intro}
	Throughout this paper, let $\N$ denote the set of positive integers and let $\N_0=\N\cup\{0\}$. For a real number $x$, we let $\Vert x\Vert$ denote the distance from $x$ to the nearest integer. 
	We say that the non-zero complex numbers
	$\alpha_1,\ldots,\alpha_n$ are multiplicatively dependent if there exists a non-zero vector 
	$(m_1,\ldots,m_n)\in\Z^n$ such that
	$\alpha_1^{m_1}\cdots\alpha_n^{m_n}=1$; otherwise they are called multiplicatively independent. Let $\lambda\in (0,1)\setminus\{1/2\}$ and let $\mu_\lambda$ be the associated Bernoulli convolution, namely the  distribution of $\displaystyle\sum_{n=0}^{\infty}\pm\lambda^n$ where the signs are chosen independently with probability $1/2$. For a survey of results in the 20th century, we refer the reader to \cite{PSS00}. The Fourier transform of $\mu_{\lambda}$ is:
	$$\widehat{\mu_\lambda}(u)=\int_{-\infty}^{\infty} e^{itu}\,d\mu_{\lambda}(t)=\prod_{n=0}^{\infty}\cos(\lambda^n u).$$
		In 1939, Erd\H{o}s \cite{Erdos1939} proved that if $\lambda^{-1}$ is a Pisot number then $\widehat{\mu_\lambda}(u)$ does not tend to $0$ as $u$ tends to infinity. The converse statement was proved by Salem \cite{Salem1943} in 1943.
	   
	   Many problems and results mentioned in this paper are at the confluence of three important research directions each of which has, on average, over 100 years of development. The first direction is the theory of sets of uniqueness and sets of multiplicity for trigonometric series. This  started  with work of Cantor \cite{Cantor1870,Cantor1872} in the 1870s and culminated in Salem's landmark result
	   \cite{Salem1943} that for $0<\lambda<1/2$, the set $E(\lambda):=\Supp(\mu_\lambda)$ is a set of uniqueness if and only if $\lambda^{-1}$ is Pisot. The second direction started from work of 
	   Erd\H{o}s, Jessen, Kershner, and Wintner in the 1930s \cite{JW1935,KW1935,Wintner1935,Erdos1939} concerning the absolute continuity of
	   $\mu_{\lambda}$. The third direction is the study of diophantine properties of the sequence $(\Vert s_n\Vert)_{n}$ for certain special sequences $(s_n)_n$. This problem featured in the 1912 ICM report by Hardy and Littlewood \cite{HL1913} with motivation from much earlier work concerning estimates of trigonometric sums in analytic number theory. In \cite{HL1913}, the authors mainly dealt with the case when $s_n$ is a polynomial function in $n$ and noted, by specific examples, that the case of exponential sequences such as $s_n=\lambda^n u$ should yield very different results. This culminated in work of Pisot \cite{Pisot1938} and Salem \cite{Salem1945} giving rise to what are nowadays called Pisot numbers and Salem numbers.

     We will not attempt to give a proper introduction to these directions or an adequate survey of recent results. Instead, we provide a concise version of Salem's list of four unsolved problems in \cite{Salem1963} with an emphasis on the third problem which will be solved in this paper and make some related comments.
     \begin{Problems}[Salem's list of 4 unsolved problems]
     \leavevmode
     \begin{itemize}
     		\item [1.] Do there exist a transcendental number $\theta>1$ and a positive number $u$ such that
     		$\displaystyle\lim_{n\to\infty}\Vert\theta^nu\Vert=0$?
     		
     		\smallskip
     		
     		\item [2.] Determine the closure of the set of Salem numbers.
     		
     		\smallskip
     		
     		\item [3.] It is known that $\displaystyle\lim_{u\to\infty}\widehat{\mu_\lambda}(u)=0$ if and only if
     		$\lambda^{-1}$ is not a Pisot number \cite{Erdos1939,Salem1943}. Now let $\lambda_1,\lambda_2\in(0,1)\setminus\{1/2\}$ such that both $\lambda_1^{-1}$ and $\lambda_2^{-1}$ are Pisot numbers. \footnote{The precise quote from Salem's book using our notation is as follows. ``What is the behavior of the product $\widehat{\mu_{\lambda_1}}(\pi u)\cdot\widehat{\mu_{\lambda_2}}(\pi u)$ as $u\to\infty$? Can this product tend to zero? Example, $\lambda_1=\frac{1}{5}$, $\lambda_2=\frac{1}{7}$.'' The first question is rather broad and the second question clarifies exactly what Salem wanted to ask.}When
     		do we have $\displaystyle\lim_{u\to\infty} \widehat{\mu_{\lambda_1}}(u)\cdot\widehat{\mu_{\lambda_2}}(u)=0$? 
     		 As explained in \cite{KS1958} and \cite[p.~62]{Salem1963}, when this happens and $\lambda_1$ and $\lambda_2$ are sufficiently small then the set $E(\lambda_1)+E(\lambda_2)$ is a set of multiplicity.
     		
     		\smallskip
     		
     		\item [4.] Determine the values of $\lambda\in (1/2,1)$ for which $\mu_{\lambda}$ is absolutely continuous.
     \end{itemize}
     \end{Problems}
      
      Problem 1 was first proposed by Hardy \cite{Hardy1919} in 1919 and it also 
      appeared as the first open problem in \cite[Chapter 10]{Bug12_DM}. It seems to be Salem's favorite: it appears at least twice in his book \cite{Salem1963} and again in his other work. Problem 4 came from work of Wintner and others in the 1930s as mentioned above. To our best knowledge, the remaining two problems were originally proposed by Salem. 
      
      Problem 1 and Problem 2 are wide open. In fact, even the special case of Problem 2 stating that the closure of the set of Salem numbers does not contain $1$ (equivalently, there is a lower bound $c>1$ for \emph{all} Salem numbers) remains a major open problem in algebraic number theory. This is also a special case of the famous Lehmer's conjecture \cite{Lehmer1933}.  Problem 4, on the other hand, has seen significant progress over the years. 
      In 1995, Solomyak \cite{Solomyak1995} proved that $\mu_{\lambda}$ is absolutely continuous with an $L^2$ density for almost all $\lambda\in (1/2,1)$. In 2014, Shmerkin \cite{Shmerkin2014} established absolute continuity outside a set of Hausdorff dimension $0$ using Hochman's work \cite{Hochman2014}. In 2019, Shmerkin \cite[Theorem~1.3(i)]{Shmerkin2019} upgraded this into having $L^q$ density for all finite $q>1$. In a ``transversal direction'' to Shmerkin's work (see the comments in \cite[p.~326]{Shmerkin2019}), Varj\'u \cite{Varju2019_absolute}
      proved a remarkable result that yields the absolute continuity of $\mu_{\lambda}$ for every algebraic number
      $\lambda\in (1-c',1)$ where $c'$ is an absolute constant. A closely related problem is to determine the dimension of $\mu_\lambda$. Some recent significant results are obtained by Breuillard-Varj\'u \cite{BV2019,BV2020} and
      Varj\'u \cite{Varju2019_dimension} which use, again, the breakthrough work of Hochman \cite{Hochman2014}.
      
      In this paper, we solve Salem's third problem. To follow Salem\footnote{Salem's book used the notation $\Gamma$. We replace $\Gamma$ by $G$ to avoid any potential confusion with the Gamma function.} and emphasize the connection to diophantine properties of the sequence $(\Vert\lambda^n u\Vert)_n$, from now on we work with the function:
      $$G_{\lambda}(u):=\widehat{\mu_\lambda}(\pi u)=\prod_{n=0}^{\infty}\cos(\pi\lambda^n u).$$
      \begin{theorem}\label{thm:solving Salem}
      Let $\lambda_1,\lambda_2\in (0,1)\setminus\{1/2\}$ such that both $\lambda_1^{-1}$ and $\lambda_2^{-1}$
      are Pisot numbers. Then $\displaystyle\lim_{u\to\infty} G_{\lambda_1}(u)\cdot G_{\lambda_2}(u)=0$
      if and only if $\lambda_1$ and $\lambda_2$ are multiplicatively independent. 
      \end{theorem}
      
      As explained below, the ``only if'' part follows immediately from \cite{Erdos1939}. Hence the whole difficulty 
      in proving Theorem~\ref{thm:solving Salem} lies in proving the ``if'' part. Roughly speaking, while it is 
      possible to 
      have many (i.e.~the cardinality of the continuum) pairs of separate sequences of values of $u$ approaching infinity, say $(u^{(1)}_n)_n$ and $(u^{(2)}_n)_n$, such that the $G_{\lambda_i}(u^{(i)}_n)$'s stay away from $0$ 
      for $i\in\{1,2\}$,
      Theorem~\ref{thm:solving Salem} confirms the heuristic that if $\lambda_1$ and $\lambda_2$ are multiplicatively independent then the multiplication-by-$\lambda_1$ and multiplication-by-$\lambda_2$ maps do not have a ``common structure'' so that it is impossible to have a \emph{common} sequence of values of $u$ approaching infinity for which both the $G_{\lambda_1}(u)$'s and $G_{\lambda_2}(u)$'s stay away from $0$. This heuristic is similar to the one underlying many conjectures made by Furstenberg in the 1960s, see \cite[Section~1.1]{Shmerkin2019} for more details. 
      
      To increase potential applications in fractal geometry, we will prove a more general result that allows functions of the form $G_{\lambda}(f(u))$ such as $G_{\lambda}(au+b)$, $G_{\lambda}(\sqrt[3]{au^2+bu+c})$, etc. where $a>0$ is an algebraic number. First, we need the following:
      
      \begin{definition}\label{def:algebraic main term}
      Let $C>0$. A function $f:\ (C,\infty)\rightarrow \mathbb{R}$ is said to have an algebraic main term if 
      there exist a rational number $r>0$, an algebraic number $\alpha>0$, and a real number $\epsilon>0$ such that
      $$f(u)=\alpha u^r+O(u^{r-\epsilon})$$
      for every $u\in (C,\infty)$.
      \end{definition}
      
      \begin{example}
      Consider a function of the form $\sqrt{au^3+bu^2+cu+d}$ where $a>0$ is algebraic and $b$, $c$, and $d$ are arbitrary real numbers. We have a sufficiently large $C$ such that the restriction of this 
      to $(C,\infty)$ yields a function $f:\ (C,\infty)\rightarrow \mathbb{R}$ satisfying:
      $$f(u)=a^{1/2}u^{3/2}+O(u^{1/2})$$
      for every $u\in (C,\infty)$.  Therefore the function $f$ has an algebraic main term. 
      In this paper, we work with functions of the form $G_{\lambda}(f(u))$ when $u$ tends to infinity. Therefore the exact value of $C$ is not important and we only need to define the original function
      $\sqrt{au^3+bu^2+cu+d}$ for all sufficiently large $u$. 
      \end{example}
       
       \begin{remark}\label{rem:u in terms of w}
       Let $f$ be a function satisfying the conditions in Definition~\ref{def:algebraic main term}. 
       From
       $w:=f(u)=\alpha u^{r}+O(u^{r-\epsilon})$, 
       we have $\alpha u^{r}=w+O(w^{(r-\epsilon)/r})$. Hence
       $$u=\alpha^{-1/r}w^{1/r}+O(w^{(1-\epsilon)/r}).$$
       This means the multi-valued inverse $u=f^{-1}(w)$  has the algebraic main term
       $\alpha^{-1/r}w^{1/r}$ and the error term $O(w^{(1-\epsilon)/r})$.
       \end{remark}

       \begin{theorem}\label{thm:algebraic main term}
       Let $\lambda_1,\lambda_2\in (0,1)\setminus\{1/2\}$ such that both $\lambda_1^{-1}$ and $\lambda_2^{-1}$ are Pisot numbers. Let $f_1$ and $f_2$ be functions that have an algebraic main term.
       If $\lambda_1$ and $\lambda_2$ are multiplicatively independent then
       $$\lim_{u\to\infty} G_{\lambda_1}(f_1(u))\cdot G_{\lambda_2}(f_2(u))=0.$$
       \end{theorem}
       
       Theorem~\ref{thm:solving Salem} follows immediately from \cite{Erdos1939} and Theorem~\ref{thm:algebraic main term}:
       \begin{proof}[Proof of Theorem~\ref{thm:solving Salem} assuming Theorem~\ref{thm:algebraic main term}] Assume that $\lambda_1$ and $\lambda_2$ are multiplicatively dependent. For $i=1,2$, it follows from \cite{Erdos1939} that
       there exists $C_i>0$ such that
       $\vert G_{\lambda_i}(\lambda_i^{-n})\vert\geq C_i$ for every $n\in\N$.
       Therefore if $u$ is a common power (with a negative exponent) of $\lambda_1$ and $\lambda_2$ then $\vert G_{\lambda_1}(u)G_{\lambda_2}(u)\vert\geq C_1C_2$.
       
       We now assume that $\lambda_1$ and $\lambda_2$ are multiplicatively independent. By Theorem~\ref{thm:algebraic main term} in the case when $f_1$ and $f_2$ are the identity map, we have $\displaystyle\lim_{u\to\infty}G_{\lambda_1}(u)G_{\lambda_2}(u)=0$. This finishes the proof. 
       \end{proof}
       
       The rest of this paper is devoted to the proof of Theorem~\ref{thm:algebraic main term}. 
       We will use tools in diophantine approximation such as the theory of 
       Weil heights \cite{Lan62_DG,BG06_HI} and 
       the more general formulation of Roth's theorem \cite{Roth1955} by Ridout \cite{Ridout1957}
       and Lang \cite[Chapter~VI]{Lan62_DG}. In fact, all these technical ingredients for the solution of 
       Salem's third problem were available shortly before the publication of his book.
       
       We conclude this section with a  brief explanation of the two main steps in the proof of Theorem~\ref{thm:algebraic main term}. To simplify the discussion at the moment, we assume that $f_1$ and $f_2$ are the identity function. Suppose $G_{\lambda_1}(u)G_{\lambda_2}(u)$ does not tend to $0$ as $u$ tends to infinity. This means there exist $C>0$ and a sequence $(u_n)_{n}$ approaching infinity such that $\vert G_{\lambda_i}(u_n)\vert\geq C$ for $i\in\{1,2\}$ and for every $n$. Let $\theta_1=\lambda_1^{-1}$,
        $\theta_2=\lambda_2^{-1}$, and write
       $$u_n=a_n\theta_1^{e_1(n)}=b_n\theta_2^{e_2(n)}$$
       where $e_1(n),e_2(n)\in\N$, $1\leq a_n<\theta_1$, and $1\leq b_n<\theta_2$.
       
       The first step is to approximate the real number $a_n$ by an algebraic number $\alpha_n\in\Q(\theta_1)$ when $n$ is sufficiently large.
      This approximation is strong in the sense that the Weil height 
      $H(\alpha_n)$ is very small compared to $\vert a_n-\alpha_n\vert^{-1}$ and $u_n$.
      Likewise, we can have a strong approximation $\beta_n\in\Q(\theta_2)$ for $b_n$. 
      This allows us to approximate $\displaystyle 1=\frac{a_n\theta_1^{e_1(n)}}{b_n\theta_2^{e_2(n)}}$ by 
      the algebraic numbers 
      $\displaystyle\frac{\alpha_n\theta_1^{e_1(n)}}{\beta_n\theta_2^{e_2(n)}}$. 
      Then in the second step, we apply Roth's theorem and arrive at a contradiction.
      Although the general idea for the proof can be explained simply like this, in practice, we need to make careful and delicate choices of the involving parameters to carry out the idea successfully. Although the general case when $f_1$ and $f_2$
      are arbitrary functions having an algebraic main term is more technical, it can be resolved following a similar approach.

     \section{Absolute values, Weil heights, and Roth's theorem}  
      Let $M_{\Q}=M_{\Q}^{\infty}\cup M_{\Q}^0$ where
  	$M_{\Q}^0$ is the set of $p$-adic absolute values
  	and $M_{\Q}^{\infty}$
  	is the singleton consisting of the usual archimedean 
  	absolute value. More generally, for every number field $K$,
  	write $M_K=M_K^\infty\cup M_K^0$ where
  	$M_K^\infty$ is the set of archimedean places and 
  	$M_K^0$ is the set of finite places. 
  	For every
  	$w\in M_K$, let $K_w$ denote the completion of 
  	$K$ with respect to $w$ and denote
  	$d_w=[K_w:\Q_v]$ where $v$ is the restriction of
  	$w$ to $\Q$. For every $v\in M_{\Q}$, we have:
    $$\sum_{w\in M_K,\,v\mid w}d_w=[K:\Q].$$
    
    Following \cite[Chapter~1]{BG06_HI}, 
  	for every $w\in M_K$ restricting to $v$ on $\Q$,
  	we normalize $\vert \cdot\vert_w$ as follows:
  	$$\vert x\vert_w = \vert \Norm_{K_w/\Q_v}(x) \vert_v^{1/[K:\Q]}.$$ 
      We note that this normalization in \cite{BG06_HI} is different from the one in \cite{Lan62_DG}. Throughout the paper, we use $\vert\cdot\vert$ to denote the usual absolute value on $\C$ and its restriction to $K$. Hence if
      $w\in M_K$ is the corresponding archimedean place then our normalization gives:
      \begin{equation*}
\vert x\vert_w= 
\begin{cases} 
  \vert x\vert^{1/[K:\Q]}, & \text{if } K_w=\R \\
  \vert x\vert^{2/[K:\Q]}, & \text{if } K_w=\C 
\end{cases}
\end{equation*}  
    For every $x\in K^*$, we have the product formula:
	$$\prod_{v\in M_K}\vert x\vert_v=1.$$
	
    We define the absolute multiplicative Weil height $H:\ \Qbar\rightarrow \R_{\geq 1}$ as follows. Let $x\in\Qbar$, choose a number field $K$ containing $x$, and define
    $$H(x)=\prod_{v\in M_K}\max\{1,\vert x\vert_v\}.$$
    This is independent of the choice of $K$ \cite[Chapter~1]{BG06_HI}. We will also use the logarithmic height
    $h:\ \Qbar\rightarrow \R_{\geq 0}$ where $h(x)=\log(H(x))$.
    
    \begin{proposition}\label{prop:height basic}
    Let $x,y\in \Qbar$. We have:
    \begin{itemize}
    \item [(i)] $H(xy)\leq H(x)H(y)$ and $H(x^m)=H(x)^{\vert m\vert}$ for every integer $m$. If $y\neq 0$, we have $H(x/y)\geq H(x)H(y)^{-1}$.
    \item [(ii)] $H(x)\geq 1$ and the equality occurs iff $x=0$ or $x$ is a root of unity.
    \item [(iii)] Northcott property: there are only finitely many algebraic numbers of bounded degree and bounded height.
    \end{itemize}
    \end{proposition}
    \begin{proof}
    All these properties are well-known. Part (i) follows immediately from the definition. The proof for the other parts can be found in \cite{BG06_HI}.
    \end{proof}
    
    \begin{proposition}\label{prop:Minkowski}
    Let $n$ be a positive integer.
    \begin{itemize}
    \item [(i)] For $x=(x_1,\ldots,x_n)\in \Z^n$, define $\vert x\vert=\vert x_1\vert+\cdots+\vert x_n\vert$. Let $B>0$ and $\psi:\ \Z^n\rightarrow\R$ with the following properties:
    		\begin{itemize}
    			\item[$\bullet$] $\psi(0)=0$ and $\psi(x)\geq B$ for $x\neq 0$.
    			\item[$\bullet$] $\psi(mx)=\vert m\vert\psi(x)$ for $m\in\Z$ and $x\in \Z^n$.
    			\item[$\bullet$] $\psi(x+y)\leq \psi(x)+\psi(y)$ for $x,y\in\Z^n$.
    		\end{itemize}
    		Then there exists $C>0$ such that $\psi(x)\geq C\vert x\vert$ for every $x\in \Z^n$.
    \item [(ii)] Let $\alpha_1,\ldots,\alpha_n\in\Qbar^*$ be multiplicatively independent. Then there exists $C>1$
    such that
    $$H(\alpha_1^{m_1}\cdots\alpha_n^{m_n})\geq C^{\vert m_1\vert+\cdots+\vert m_n\vert}$$
    for every $(m_1,\ldots,m_n)\in\Z^n$.
    \end{itemize}
    \end{proposition}
    \begin{proof}
    Part (i) should be well-known to Minkowski when he developed his theory Geometry of Numbers in the late 19th century while part (ii) should be well-known to experts in the 1950s or even earlier. However we are unable to locate precise references that far back. For a proof of (i), see \cite[pp.~198--199]{Sch03_LR}. For (ii),
    define $\psi:\ \Z^n\rightarrow\mathbb{R}$ by
    $$\psi(m_1,\ldots,m_n)=h(\alpha_1^{m_1}\cdots\alpha_n^{m_n}).$$
    This function satisfies all the properties of part (i) thanks to Proposition~\ref{prop:height basic}: the existence of a uniform lower bound $B>0$ follows from
    the Northcott property and
     the fact that $\alpha_1^{m_1}\cdots\alpha_n^{m_n}$ is not a root of unity when 
    $(m_1,\ldots,m_n)$ is non-zero. Then we apply part (i) and finish the proof.
    \end{proof}
    
    The final ingredient for the proof of Theorem~\ref{thm:algebraic main term} is Lang's general formulation of Roth's theorem \cite{Roth1955} (the same formulation for the special case of rational numbers was given earlier by Ridout \cite{Ridout1957}).
    \begin{theorem}\label{thm:Roth}
    Let $K$ be a number field with a finite set of places $S$. For each $v\in S$, let $\alpha_v$ be either
    an element of $K_v$ that is algebraic over $K$ or the symbol $\infty$. Let $\kappa>2$. Then there are only finitely many $\beta\in K$ such that 
    $$\prod_{v\in S}\min(1,\vert \beta-\alpha_v\vert_v)\leq H(\beta)^{-\kappa};$$
    when $\alpha_v$ is the symbol $\infty$, we interpret $\vert\beta-\alpha_v\vert_v$ as $1/\vert\beta\vert_v$.
    \end{theorem}
    \begin{proof}
	See Theorem~6.2.3 and Remark~6.2.5 in  \cite{BG06_HI}.    
    \end{proof}

    \section{A diophantine approximation result}\label{sec:diophantine}
	Throughout this section, let $\lambda_1,\lambda_2\in (0,1)\setminus\{1/2\}$ such that $\theta_1:=\lambda^{-1}$ and $\theta_2:=\lambda_2^{-1}$
	are Pisot numbers. Let $B>1$. We assume that
     $(r_n)_n$ and 
     $(s_n)_n$ are
     sequences approaching infinity such that
	\begin{equation}\label{eq:B}
    \delta:=\inf_n \vert G_{\lambda_1}(r_n)G_{\lambda_2}(s_n)\vert >0\ \text{and}\ r_n^{1/B}\leq s_n\leq r_n^B\ \text{for every $n$.}
    \end{equation}
    In the following $C_0,C_1,\ldots$ denote positive constants depending only on $\lambda_1$, $\lambda_2$, $\delta$, and $B$. We fix a large integer $N$ while $n$ will be an arbitrarily large integer.
	
    Note that $\delta\leq 1$. When $n$ is sufficiently large (so that $r_n\geq \theta_1$ and $s_n\geq \theta_2$), we write
	$$r_n=a_n\theta_1^{e_1(n)}\ \text{and}\ s_n=b_n\theta_2^{e_2(n)}$$
	where $e_1(n),e_2(n)\in\N$, $1\leq a_n<\theta_1$, and $1\leq b_n<\theta_2$. Thanks to \eqref{eq:B}, we have $C_0>1$ such that
    \begin{equation}\label{eq:C0}
        \frac{1}{C_0}<\frac{e_1(n)}{e_2(n)}<C_0.
    \end{equation}
    We have
	\begin{align}\label{eq:RHS for Glambdarnsn}
	\begin{split}
    \vert G_{\lambda_1}(r_n)\vert&=\prod_{-\infty<j\leq e_1(n)}\vert\cos(\pi a_n\theta_1^j)\vert\\
    &=\vert \cos(\pi a_n\theta_1^{e_1(n)})\vert\cdot\vert \cos(\pi a_n\theta_1^{e_1(n)-1})\vert\cdots\\
    \vert G_{\lambda_2}(s_n)\vert&=\prod_{-\infty<j\leq e_2(n)}\vert\cos(\pi b_n\theta_2^j)\vert\\
    &=\vert \cos(\pi b_n\theta_2^{e_2(n)})\vert\cdot\vert \cos(\pi b_n\theta_2^{e_2(n)-1})\vert\cdots
	\end{split}
    \end{align}

    Let $g_1(X)$ and $g_2(X)$ be respectively the minimal polynomials of $\lambda_1$ and $\lambda_2$ over $\Q$. Let
	$C_1:=\max\{\ell^1(g_1),\ell^1(g_2)\}$ 
    where the $\ell^1$-norm of a polynomial is the sum of the absolute values of its coefficients.
     Let $C_2\geq 2$ be a sufficiently large positive integer such that $\delta^{1/(C_2-1)}$ is sufficiently close to $1$ so that the following holds:
		\begin{equation}\label{eq:C2}
		\text{if $x$ is a real number such that $\vert\cos(\pi x)\vert\geq \delta^{1/(C_2-1)}$ then $\Vert x\Vert<\frac{1}{C_1}$.}
		\end{equation}
	
	\begin{proposition}\label{prop:k ell exist}
	When $n$ is sufficiently large, there exist integers $k$ and $\ell$ 
    with the following properties:
    \begin{itemize}
        \item [(a)] $\lfloor e_1(n)/N^{C_2}\rfloor\leq k\leq e_1(n)/N^2$
        \item [(b)] $\lfloor e_2(n)/N^{C_2}\rfloor\leq \ell\leq e_2(n)/N^2$
        \item [(c)] $\displaystyle\Vert a_n\theta_1^j\Vert<\frac{1}{C_1}$ for $k\leq j\leq Nk-1$.
        \item [(d)] $\displaystyle\Vert b_n\theta_2^j\Vert<\frac{1}{C_1}$ for $\ell\leq j\leq N\ell-1$.
        \item [(e)] $\displaystyle\frac{1}{C_3}<\frac{k}{\ell}<C_3$.
    \end{itemize}
	\end{proposition}
	\begin{proof}	
		We assume $n$ is sufficiently large so that $\min\{e_1(n),e_2(n)\}\geq N^{C_2}$. 
        For $1\leq i\leq C_2$, let $k_i=\lfloor e_1(n)/N^i\rfloor$ and $\ell_i=\lfloor e_2(n)/N^i\rfloor$. In other words, the increasing sequence
		$k_{C_2},k_{C_2-1},\ldots,k_1$ is exactly
		$$\left\lfloor\frac{e_1(n)}{N^{C_2}}\right\rfloor,\left\lfloor\frac{e_1(n)}{N^{C_2-1}}\right\rfloor,\ldots,\left\lfloor\frac{e_1(n)}{N}\right\rfloor$$
		and the increasing sequence 
        $\ell_{C_2},\ell_{C_2-1},\ldots,\ell_1$ is exactly
        $$\left\lfloor\frac{e_2(n)}{N^{C_2}}\right\rfloor,\left\lfloor\frac{e_2(n)}{N^{C_2-1}}\right\rfloor,\ldots,\left\lfloor\frac{e_2(n)}{N}\right\rfloor.$$
        
        For $2\leq i\leq C_2$, let
		$P_i=\displaystyle\prod_{j=k_i}^{k_{i-1}-1}\vert\cos(\pi a_n\theta_1^j)\vert$		
		and $Q_i=\displaystyle\prod_{j=\ell_i}^{\ell_{i-1}-1}\vert\cos(\pi b_n\theta_2^j)\vert$
        so that
		$$P_2\cdots P_{C_2}Q_2\cdots Q_{C_2}=\prod_{j=\lfloor e_1(n)/N^{C_2}\rfloor}^{\lfloor e_1(n)/N\rfloor-1} \vert\cos(\pi a_n\theta_1^j)\vert\cdot \prod_{j=\lfloor e_2(n)/N^{C_2}\rfloor}^{\lfloor e_2(n)/N\rfloor-1} \vert\cos(\pi b_n\theta_2^j)\vert$$		
		  is a partial product of $\vert  G_{\lambda_1}(r_n)\vert\cdot\vert G_{\lambda_2}(s_n)\vert$ thanks to \eqref{eq:RHS for Glambdarnsn}. Since 
		$\delta\leq \vert G_{\lambda_1}(r_n)\vert\cdot\vert G_{\lambda_2}(s_n)\vert$, we have $P_2\cdots P_{C_2}Q_2\cdots Q_{C_2}\geq \delta$ and hence 
		there exists $i\in\{2,\ldots,C_2\}$ such that
		$P_iQ_i\geq \delta^{1/(C_2-1)}$. This implies
		$$\vert\cos(\pi a_n\theta_1^j)\vert \geq\delta^{1/(C_2-1)}\ \text{and hence}\ \Vert a_n\theta_1^j\Vert<\frac{1}{C_1}\ \text{because of \eqref{eq:C2}}$$
		for $k_i\leq j\leq k_{i-1}-1$ and implies
        $$\vert\cos(\pi b_n\theta_2^j)\vert \geq\delta^{1/(C_2-1)}\ \text{and hence}\ \Vert b_n\theta_2^j\Vert<\frac{1}{C_1}\ \text{because of \eqref{eq:C2}}$$
		for $\ell_i\leq j\leq \ell_{i-1}-1$.
        
        Since $k_{i-1}\geq Nk_i$ and $\ell_{i-1}\geq N\ell_i$, we take $k:=k_i$ and $\ell:=\ell_i$ and have that properties
        (a)--(d) hold. It remains to prove property (e).
        From $\displaystyle\frac{k}{\ell}=\frac{k_i}{\ell_i}=\frac{\lfloor e_1(n)/N^i\rfloor}{\lfloor e_2(n)/N^i\rfloor}$, we have
        $$\frac{1}{2}\cdot\frac{e_1(n)}{e_2(n)}<\frac{k}{\ell}<2\cdot\frac{e_1(n)}{e_2(n)}$$
        when $n$ is sufficiently large. Then property (e) follows from
        \eqref{eq:C0}.
	\end{proof}
	
	For the rest of this section, we let $k$ and $\ell$ be as in the statement of Proposition~\ref{prop:k ell exist}. 
    We write the minimal polynomial $g_1(X)$ of $\theta_1$ as:
    $$g_1(X)=X^d+c_{d-1}X^{d-1}+\cdots+c_1X+c_0.$$
    For $k\leq j\leq Nk-1$, let $A_j$ be the nearest integer to $a_n\theta_1^j$, and let $\epsilon_j=a_n\theta_1^j-A_j$ so that 
	\begin{equation}\label{eq:Aj and Bj}	
	a_n\theta_1^j=A_j+\epsilon_j\ \text{and}\ \vert \epsilon_j\vert=\Vert a_n\theta_1^j\Vert<\frac{1}{C_1}.
	\end{equation}
	
	\begin{lemma}\label{lem:linear recurrence}
	The sequence 	$(A_j)_{k\leq j\leq Nk-1}$ satisfies the linear recurrence relation
	\begin{equation}
	    A_{j+d}+c_{d-1}A_{j+d-1}+\cdots+c_1A_{j+1}+c_0A_j=0
	\end{equation}
    for $k\leq j\leq Nk-1-d$.
    \end{lemma}
    \begin{proof}
        By \eqref{eq:Aj and Bj}, the definition of $C_1$ (which gives
        $C_1\geq 1+\vert c_{d-1}\vert+\cdots+\vert c_0\vert$), and the fact that $\theta_1^d+c_{d-1}\theta_1^{d-1}+\cdots+c_0=0$, we have:
        $$\vert A_{j+d}+c_{d-1}A_{j+d-1}+\cdots+c_0A_j\vert=\vert \epsilon_{j+d}+c_{d-1}\epsilon_{j+d-1}+\cdots+c_0\epsilon_j\vert<1$$
        for $k\leq j\leq Nk-1-d$. 
        Since the LHS is an integer, it has to be zero.
    \end{proof}

    Let $t_1=\theta_1,t_2,\ldots,t_d$ be the roots of $g_1(X)$ (i.e.~the Galois conjugates of $\theta_1$). Our next step is to find $\alpha_1,\ldots,\alpha_d$ such that
    \begin{equation}\label{eq:alpha1...alphad}
        A_{j}=\alpha_1t_1^j+\cdots+\alpha_d t_d^j\ \text{for $k\leq j\leq Nk-1$}
    \end{equation}
    Since both sides of \eqref{eq:alpha1...alphad} satisfy the same linear recurrence relation with the companion polynomial $g_1(X)$, it suffices to find $\alpha_1,\ldots,\alpha_d$ such that \eqref{eq:alpha1...alphad} holds for $j=k,\ldots,k+d-1$. We need to solve the linear system
    \begin{equation}
        M(\alpha_1,\ldots,\alpha_d)^T=(A_k,\ldots,A_{k+d-1})^T
    \end{equation}
    where 
    $$M=\begin{pmatrix}
t_1^k & t_2^k & \cdots & t_d^k \\
t_1^{k+1} & t_2^{k+1} & \cdots & t_d^{k+1} \\
\vdots & \vdots & \ddots & \vdots \\
t_1^{k+d-1} & t_2^{k+d-1} & \cdots & t_d^{k+d-1}
\end{pmatrix}$$

    Note that $t_1\cdots t_d=(-1)^dc_0$. Let $\Delta:=\displaystyle\prod_{1\leq i<j\leq d}(t_i-t_j)$. Then the Vandermonde determinant formula gives
    \begin{equation}
        \det(M)=(-1)^{dk}c_0^k\Delta
    \end{equation}
    For $1\leq i\leq d$, let $M_i$ be the matrix obtained by replacing the $i$-th column of $M$ by $(A_k,\ldots,A_{k+d-1})^T$. Cramer's rule gives:
    \begin{equation}\label{eq:Cramer}
        \alpha_i=\frac{\det(M_i)}{\det(M)}=\frac{\det(M_i)}{(-1)^{dk}c_0^k\Delta}
    \end{equation}
    
    \begin{proposition}\label{prop:inequalities using Hadamard}
        We have the following inequalities:
        \begin{itemize}
            \item [(i)] $\vert \det(M_1)\vert\leq C_4\theta_1^k$ and $\vert\det(M_i)\vert\leq C_4\theta_1^{2k}$ for $2\leq i\leq d$.
            
            \item [(ii)] $\vert \alpha_1\vert\leq C_4\theta_1^k$ and $\vert\alpha_i\vert\leq C_4\theta_1^{2k}$ for $2\leq i\leq d$.

            \item [(iii)] $\displaystyle\vert a_n-\alpha_1\vert\leq\frac{C_5}{\theta_1^{(N-2)k}}$. 
        \end{itemize}
    \end{proposition}
    \begin{proof}
        From $A_j< a_n\theta_1^j+1< \theta_1^{j+1}+1$ and the fact that $t_1=\theta_1$ is a Pisot number,  
        we apply Hadamard's inequality for $\det(M_1)$ and $\det(M_i)$ for $2\leq i\leq d$ to finish the proof of part (i). Since $c_0$ and $\Delta^2$ are non-zero integers, 
        \eqref{eq:Cramer} implies
        $\vert\alpha_i\vert\leq\vert\det(M_i)\vert$ for $1\leq i\leq d$.
        Then we apply part (i) to conclude part (ii).

        For part (iii), we have:
        \begin{align*}
            a_n\theta^{Nk-1}&=A_{Nk-1}+B_{Nk-1}\\
            &=\alpha_1\theta^{Nk-1}+\alpha_2t_2^{Nk-1}+\cdots+\alpha_dt_d^{Nk-1}+B_{Nk-1}
        \end{align*}
        Dividing both sides by $\theta^{Nk-1}$ and using part (ii), we obtain the desired inequality in part (iii).
    \end{proof}

    In fact, we can bound $\alpha_1$ uniformly for every sufficiently large $n$:
    \begin{corollary}
    When $n$ is sufficiently large, we have $\displaystyle\frac{1}{2}<\alpha_1<\theta_1+\frac{1}{2}$.
    \end{corollary}
    \begin{proof}
        By (iii) of the previous proposition, we have
        $$\vert a_n-\alpha_1\vert\leq \frac{C_5}{\theta_1^{(N-2)k}}<\frac{1}{2}$$
        when $n$ is sufficiently large. Since $1\leq a_n<\theta_1$, we finish the proof.
    \end{proof}
    
    \begin{lemma}\label{lem:Galois conjugates of alphai}
        Let $\sigma\in\Gal(\Qbar/\Q)$. For $1\leq i,j\leq d$, if $\sigma(t_i)=t_j$ then
        $\sigma(\alpha_i)=\alpha_j$. Consequently, for every $1\leq i\leq d$, we have $\alpha_i\in\Q(t_i)$ and the set $\{\alpha_1,\ldots,\alpha_d\}$ (after removing repeated elements) is exactly the set of all Galois conjugates of $\alpha_i$.
    \end{lemma}
    \begin{proof}
        We have already seen that the linear system 
        \begin{equation}\label{eq:linear system}
        M(x_1,\ldots,x_d)^T=(A_k,\ldots,A_{k+d-1})^T
        \end{equation}
        has the unique solution $(\alpha_1,\ldots,\alpha_d)$.
        By applying $\sigma$ to the equation
        $$M(\alpha_1,\ldots,\alpha_d)^T=(A_k,\ldots,A_{k+d-1})^T$$
        and rearranging the columns on the LHS, we have that the system \eqref{eq:linear system} has a solution of the form
        $(\alpha_1',\ldots,\alpha_d')$ where $\alpha_j'=\sigma(\alpha_i)$.
        Since $(\alpha_1',\ldots,\alpha_d')=(\alpha_1,\ldots,\alpha_d)$,
        we have $\alpha_j=\sigma(\alpha_i)$. The remaining assertions follow immediately from the first assertion.
    \end{proof}

    \begin{corollary}\label{cor:Galois conjugates of detMi}
        Let $\sigma\in\Gal(\Qbar/\Q)$. For $1\leq i,j\leq d$, if $\sigma(t_i)=t_j$ then $\sigma(\det(M_i))=\pm\det(M_j)$.
    \end{corollary}
    \begin{proof}
        Lemma~\ref{lem:Galois conjugates of alphai} gives $\sigma(\alpha_i)=\alpha_j$. Then we use \eqref{eq:Cramer}
        and the fact that $\sigma(\Delta)=\pm \Delta$ since $\Delta^2\in\Z$ to finish the proof.
    \end{proof}
    
    \begin{proposition}\label{prop:H alpha1}
        $H(\alpha_1)\leq C_4\theta_1^{2k}$
    \end{proposition}
    \begin{proof}
        Let $K=\Q(\alpha_1,\Delta)$. Since $c_0$ and $\Delta^2$ are integers, we have
        $\vert c_0^k\Delta\vert_v=\vert c_0^k\Delta\vert^{d_v/[K:\Q]}$ for
        every $v\in M_K$.
        By Proposition~\ref{prop:inequalities using Hadamard} and Corollary~\ref{cor:Galois conjugates of detMi}, all the Galois conjugates of $\det(M_1)$ have moduli at most 
        $C_4\theta_1^{2k}$. Therefore
        $\vert\det(M_1)\vert_v\leq (C_4\theta_1^{2k})^{d_v/[K:\Q]}$ for every $v\in M_K^{\infty}$. Combining with \eqref{eq:Cramer} and using $\vert c_0\vert\leq\theta_1$, we have:
        $$\max\{1,\vert\alpha_1\vert_v\}\leq \max\{1,(C_4\theta_1^{2k})/(\vert c_0^k\Delta\vert)\}^{d_v/[K:\Q]}=\left(\frac{C_4\theta_1^{2k}}{\vert c_0^k\Delta\vert}\right)^{d_v/[K:\Q]}$$
        for every $v\in M_K^{\infty}$, where the last equality
        follows from
        $$\frac{C_4\theta_1^{2k}}{\vert c_0^k\Delta\vert}\geq \frac{C_4\theta^{k}}{\vert\Delta\vert}\geq 1$$
        when $n$ is sufficiently large. Therefore
        \begin{equation}\label{eq:Halpha1 archimedean}
        \prod_{v\in M_K^\infty}\max\{1,\vert\alpha\vert_v\}\leq C_4\theta_1^{2k}\prod_{v\in M_K^{\infty}}\vert 1/(c_0^k\Delta)\vert_v.
        \end{equation}
        Since $\det(M_1)$ is an algebraic integer, we have
        \begin{equation}\label{eq:Halpha1 finite}
        \max\{1,\vert \alpha_1\vert_v\}\leq \vert 1/(c_0^k\Delta)\vert_v
        \end{equation}
        for every $v\in M_K^0$. From \eqref{eq:Halpha1 archimedean}, \eqref{eq:Halpha1 finite}, and the product formula, we have
        $$H(\alpha_1)=\prod_{v\in M_K}\max\{1,\vert\alpha_1\vert_v\}\leq C_4\theta_1^{2k}.$$
    \end{proof}
       
    Proposition~\ref{prop:inequalities using Hadamard}(iii) and Proposition~\ref{prop:H alpha1} show that $a_n$ can be (strongly) approximated by the algebraic number $\alpha_1\in\Q(\theta_1)$; the strength of this approximation depends on how large $N$ is.
    We can now obtain similar results from Lemma~\ref{lem:linear recurrence} to Proposition~\ref{prop:H alpha1} using similar arguments
    when replacing the data $(\theta_1,a_n,k)$ by $(\theta_2,b_n,\ell)$. This gives a similar approximation to $b_n$ by an element in $\Q(\theta_2)$. While such approximation to $a_n$ and $b_n$ can be done ``separately'' like we did, it is indeed crucial to obtain $k$ and $\ell$ ``simultaneously'' to guarantee the inequalities in Proposition~\ref{prop:k ell exist}(e).
    
    We combine the key results discussed throughout this section into the following:
    \begin{theorem}\label{thm:everything}
        For every sufficiently large $n$, there exist $k,\ell\in\N$, $\alpha\in \Q(\theta_1)$, and $\beta\in\Q(\theta_2)$ (also depending on $N$ and $n$) satisfying the following properties:
            \begin{itemize}
                \item [$\bullet$] $\lfloor e_1(n)/N^{C_2}\rfloor\leq k\leq e_1(n)/N^2$, $\lfloor e_2(n)/N^{C_2}\rfloor\leq \ell\leq e_2(n)/N^2$, and $1/C_3<k/\ell<C_3$.
                \item [$\bullet$] $\vert a_n-\alpha\vert\leq C_5/\theta_1^{(N-2)k}$ and $H(\alpha)\leq C_4\theta_1^{2k}$. 
                \item [$\bullet$] $\vert b_n-\beta\vert\leq C_6/\theta_2^{(N-2)\ell}$ and $H(\beta)\leq C_7\theta_2^{2\ell}$.
                \item [$\bullet$] $1/2<\alpha<\theta_1+(1/2)$ and $1/2<\beta<\theta_2+(1/2)$.
            \end{itemize}
    \end{theorem}

    \section{Proof of Theorem~\ref{thm:algebraic main term}}
    Let $\lambda_1$, $\lambda_2$, $f_1$, and $f_2$ be as in Theorem~\ref{thm:algebraic main term}. Assume that $\lambda_1$ and $\lambda_2$ are multiplicatively independent and 
    there exist $\delta>0$ and a sequence
    $(u_n)_n$ approaching infinity such that 
    $$\vert G_{\lambda_1}(f_1(u_n))\cdot G_{\lambda_2}(f_2(u_n))\vert\geq \delta$$
    for every $n$. We will arrive at a contradiction.

    Throughout this section $C_8,C_9,\ldots$ denote positive constants depending only on $\lambda_1$, $\lambda_2$, $f_1$, $f_2$, and $\delta$. Occasionally, we use the big $O$ notation in which the implied constants follow the same dependency convention as the $C_i$'s. We fix a large positive integer $N>2$ that will be specified later while $n$ denotes an arbitrarily large integer. Let $\theta_1=\lambda_1^{-1}$ and $\theta_2=\lambda_2^{-1}$. By our assumption on $f_1$ and $f_2$, there exist positive numbers $c_1,c_2\in\Q$, $\gamma_1,\gamma_2\in\Qbar$, and $\epsilon_1,\epsilon_2\in\R$ such that
    $$\vert f_1(u)-\gamma_1u^{c_1}\vert\leq C_8u^{c_1-\epsilon_1}\ 
    \text{and}\ \vert f_2(u)-\gamma_2u^{c_2}\vert\leq C_9u^{c_2-\epsilon_2}.$$

    Let $r_n=f_1(u_n)$ and $s_n=f_2(u_n)$. Then thanks to the above asymptotic properties of $f_1$ and $f_2$, there exists $C_{10}>1$
    such that 
    $$r_n^{1/C_{10}}<s_n<r_n^{C_{10}}.$$
    Therefore the sequences $(r_n)_n$ and $(s_n)_n$ satisfy the conditions
    required in Section~\ref{sec:diophantine} (with $C_{10}$ playing the role of $B$ there). Write
    $$r_n=a_n\theta_1^{e_1(n)}\ \text{and}\ s_n=b_n\theta_2^{e_2(n)}$$
    as in Section~\ref{sec:diophantine}. As in \eqref{eq:C0}, we have $C_{11}>1$ such that
    \begin{equation}\label{eq:C11}
        \frac{1}{C_{11}}<\frac{e_1(n)}{e_2(n)}<C_{11}.
    \end{equation}
    
    Theorem~\ref{thm:everything}
    implies that for every sufficiently large $n$, there exist integers $k_n,\ell_n\in\N$, $\alpha_n\in\Q(\theta_1)$, and $\beta_n\in \Q(\theta_2)$ such that:
    \begin{align}\label{eq:from Section 3}
         \begin{split}
          & \lfloor e_1(n)/N^{C_{12}}\rfloor\leq k_n\leq e_1(n)/N^2,\ \lfloor e_2(n)/N^{C_{12}}\rfloor\leq \ell_n\leq e_2(n)/N^2,\\
         & 1/C_{13}<k_n/\ell_n<C_{13},\\
          & \vert a_n-\alpha_n\vert\leq C_{14}/\theta_1^{(N-2)k_n},\ H(\alpha_n)\leq  C_{14}\theta_1^{2k_n},\\
        & \vert b_n-\beta_n\vert\leq C_{14}/\theta_2^{(N-2)\ell_n},\ H(\beta_n)\leq C_{14}\theta_2^{2\ell_n},\\
         & 1/2<\alpha_n<\theta_1+(1/2)\ \text{and}\ 1/2<\beta_n<\theta_2+(1/2).
        \end{split}
    \end{align}
    We emphasize that although these $k_n$, $\ell_n$, $\alpha_n$, and $\beta_n$ also depend on $N$, the constants $C_i$'s and the implied ones in the big $O$ notation do not.
    
    From Remark~\ref{rem:u in terms of w}, we have
    \begin{equation}\label{eq:un in rn and sn}
        u_n=\gamma_1^{-1/c_1}r_n^{1/c_1}+ O(r_n^{(1-\epsilon_1)/c_1})\ \text{and}\ 
         u_n=\gamma_2^{-1/c_2}s_n^{1/c_2} + O(s_n^{(1-\epsilon_2)/c_2})
    \end{equation}

    Let $D$ be the common denominator of $1/c_1$ and $1/c_2$ so that both $D/c_1$ and $D/c_2$ are positive integers. From \eqref{eq:un in rn and sn}, we have:
    \begin{equation}\label{eq:un^D for theta1}
        u_n^D=\gamma_1^{-D/c_1}r_n^{D/c_1} +O(r_n^{(D-\epsilon_1)/c_1})
    \end{equation}
    and
    \begin{equation}\label{eq:un^D for theta2}
         u_n^D=\gamma_2^{-D/c_2}s_n^{D/c_2}+ O(s_n^{(D-\epsilon_2)/c_2}).
    \end{equation}
%    This gives 
%    \begin{equation}\label{eq:un^D sim}
%    \gamma_1^{-D/c_1}r_n^{D/c_1}\sim u_n^D \sim  \gamma_2^{-D/c_2}s_n^{D/c_2}\ \text{as $n\to\infty$}.
%    \end{equation}

    Our next step is to use the algebraic approximations
    $\alpha_n$ and $\beta_n$ of $a_n$ and $b_n$ to replace
    $\gamma_1^{-D/c_1}r_n^{D/c_1}$ and 
    $\gamma_2^{-D/c_2}s_n^{D/c_2}$ in \eqref{eq:un^D for theta1} and
    \eqref{eq:un^D for theta2} 
    by corresponding algebraic numbers. From $\vert a_n-\alpha_n\vert\leq C_{14}/\theta_1^{(N-2)k_n}$ and the uniform bound on $\alpha_n$ in \eqref{eq:from Section 3}, 
    we have
    $a_n^{D/c_1}=\alpha_n^{D/c_1}+O(1/\theta_1^{(N-2)k_n})$. Therefore
    \begin{align}\label{eq:replace an by alphan}
    \begin{split}
    \gamma_1^{-D/c_1}r_n^{D/c_1}&=\gamma_1^{-D/c_1}a_n^{D/c_1}\theta_1^{e_1(n)D/c_1}\\
    &=\gamma_1^{-D/c_1}\alpha_n^{D/c_1}\theta_1^{e_1(n)D/c_1}+O(\theta_1^{(e_1(n)D/c_1)-(N-2)k_n})
    \end{split}
    \end{align}
    Since $\displaystyle k_n\leq e_1(N)/N^2$, we can choose $N$ so that the $O$-term in \eqref{eq:replace an by alphan} dominates
    the one in \eqref{eq:un^D for theta1}. More precisely,
    $$\frac{r_n^{(D-\epsilon_1)/c_1}}{\theta_1^{(e_1(n)D/c_1)-(N-2)k_n}}=O(\theta_1^{(N-2)k_n-(e_1(n)\epsilon_1/c_1)})=o(1)$$
    as $n\to\infty$ if $N$ satisfies
    \begin{equation}\label{eq:N condition epsilon1/c1}
    \frac{N-2}{N^2}<\frac{\epsilon_1}{c_1}
    \end{equation}
    so that
    $$(N-2)k_n-\frac{e_1(N)\epsilon_1}{c_1}<e_1(N)\left(\frac{N-2}{N^2}-\frac{\epsilon_1}{c_1}\right)$$
    approaching $-\infty$ as $n\to\infty$. Similarly,
    \begin{equation}\label{eq:replace bn by betan}
        \gamma_2^{-D/c_2}s_n^{D/c_2}=\gamma_2^{-D/c_2}\beta_n^{D/c_2}\theta_2^{e_2(n)D/c_2}+O(\theta_2^{(e_2(n)D/c_2)-(N-2)\ell_n})
    \end{equation}
    and the $O$-term in \eqref{eq:replace bn by betan} dominates
    the $O$-term in \eqref{eq:un^D for theta2} if $N$ satisfies
    \begin{equation}\label{eq:N condition epsilon2/c2}
        \frac{N-2}{N^2}<\frac{\epsilon_2}{c_2}
    \end{equation}
    Combining \eqref{eq:un^D for theta1} and \eqref{eq:replace an by alphan} and the comparison between the $O$-terms, we have:
    \begin{equation}\label{eq:after replacing alphan}
        \vert u_n^D-\gamma_1^{-D/c_1}\alpha_n^{D/c_1}\theta_1^{e_1(n)D/c_1}\vert=O(\theta_1^{(e_1(n)D/c_1)-(N-2)k_n})
    \end{equation}
    for all sufficiently large $n$. Similarly,
    \begin{equation}\label{eq:after replacing betan}
        \vert u_n^D - \gamma_2^{-D/c_2}\beta_n^{D/c_2}\theta_2^{e_2(n)D/c_2}\vert=O(\theta_2^{(e_2(n)D/c_2)-(N-2)\ell_n}) 
    \end{equation}
    for all sufficiently large $n$. 
    
    Then we have
    $$\gamma_1^{-D/c_1}\alpha_n^{D/c_1}\theta_1^{e_1(n)D/c_1}\sim u_n^D \sim \gamma_2^{-D/c_2}\beta_n^{D/c_2}\theta_2^{e_2(n)D/c_2}$$
    as $n\to\infty$. Together with the uniform bounds on $\alpha_n$ and $\beta_n$ in \eqref{eq:from Section 3}, we have:
    \begin{equation}\label{eq:C15}
        \frac{1}{C_{15}}< \frac{\theta_1^{e_1(n)D/c_1}}{\theta_2^{e_2(n)D/c_2}}<C_{15}
    \end{equation}
    for all sufficiently large $n$. From \eqref{eq:after replacing alphan} and \eqref{eq:after replacing betan}, we have
    \begin{align}
        \begin{split}
        &\vert\gamma_1^{-D/c_1}\alpha_n^{D/c_1}\theta_1^{e_1(n)D/c_1}-\gamma_2^{-D/c_2}\beta_n^{D/c_2}\theta_2^{e_2(n)D/c_2}\vert\\
        \leq & C_{16}\left(\theta_1^{(e_1(n)D/c_1)-(N-2)k_n}+\theta_2^{(e_2(n)D/c_2)-(N-2)\ell_n}\right).
        \end{split}
    \end{align}
    Then we divide both sides by $\gamma_1^{-D/c_1}\beta_n^{D/c_2}\theta_2^{e_2(n)D/c_2}$
    and use \eqref{eq:C15} to obtain
    \begin{equation}\label{eq:C17}
        \left\vert\frac{\alpha_n^{D/c_1}\theta_1^{e_1(n)D/c_1}}{\beta_n^{D/c_2}\theta_2^{e_2(n)D/c_2}}-\frac{\gamma_2^{-D/c_2}}{\gamma_1^{-D/c_1}}\right\vert\leq C_{17}(\theta_1^{-(N-2)k_n}+\theta_2^{-(N-2)\ell_n})
    \end{equation}
    for all sufficiently large $n$. 
    We recall the bound $1/C_{13}<k_n/\ell_n<C_{13}$ in \eqref{eq:from Section 3}. This yields
    $C_{18}$ such that
    $$k_n>\frac{k_n+\ell_n}{C_{18}}\ \text{and}\ \ell_n>\frac{k_n+\ell_n}{C_{18}}.$$
    Therefore, for all sufficiently large $n$, we have
    \begin{align}
    \begin{split}    
    \text{the RHS of \eqref{eq:C17}}&<C_{17}(\theta_1^{-(N-2)(k_n+\ell_n)/C_{18}}+\theta_2^{-(N-2)(k_n+\ell_n)/C_{18}})\\
    &<C_{19}^{-(N-2)(k_n+\ell_n)}
    \end{split}
    \end{align}
    with $C_{19}>1$ (for example, any $C_{19}>1$ such that $C_{19}<\min\{\theta_1^{1/C_{18}},\theta_2^{1/C_{18}}\}$ will work). Overall, we have
    \begin{equation}\label{eq:approximation with C19}
        \left\vert\frac{\alpha_n^{D/c_1}\theta_1^{e_1(n)D/c_1}}{\beta_n^{D/c_2}\theta_2^{e_2(n)D/c_2}}-\frac{\gamma_2^{-D/c_2}}{\gamma_1^{-D/c_1}}\right\vert<C_{19}^{-(N-2)(k_n+\ell_n)}
    \end{equation}
    for all sufficiently large $n$.
    
    We now have approximations of
    $\gamma_2^{-D/c_2}/\gamma_2^{-D/c_1}$ by the algebraic numbers
    $$t_n:=\frac{\alpha_n^{D/c_1}\theta_1^{e_1(n)D/c_1}}{\beta_n^{D/c_2}\theta_2^{e_2(n)D/c_2}}$$
    and with an appropriate choice of $N$ we can apply Roth's theorem to arrive at a contradiction, as follows.

    Let $C_{20}=\max\{H(\theta_1^{D/c_1}),H(\theta_2^{D/c_2})\}$. Then we apply Proposition~\ref{prop:height basic}(i) and Proposition~\ref{prop:Minkowski} to get $C_{21}>1$
    such that
    \begin{equation}\label{eq:C20 C21}
        C_{21}^{e_1(n)+e_2(n)}\leq H(\theta_1^{e_1(n)D/c_1}/\theta_2^{e_2(n)D/c_2})\leq C_{20}^{e_1(n)+e_2(n)}.
    \end{equation}
    Proposition~\ref{prop:height basic}(i) and the upper bounds on $H(\alpha_n)$ and $H(\beta_n)$ in \eqref{eq:from Section 3} give
    \begin{equation}\label{eq:C22 C23}
        H(\alpha_n^{D/c_1}/\beta_n^{D/c_2})\leq C_{22}\theta_1^{2k_nD/c_1}\theta_2^{2\ell_nD/c_2}\leq C_{23}^{k_n+\ell_n}\leq C_{23}^{(e_1(n)+e_2(n))/N^2}
    \end{equation}
    for all sufficiently large $n$ (for example, choosing any $C_{23}>\max\{\theta_1^{2D/c_1},\theta_2^{2D/c_2}\}$ will work).
    We now assume that $N$ satisfies
    \begin{equation}\label{eq:N condition C23 C21}
        C_{23}^{1/N^2}<C_{21}.        
    \end{equation}
    Then Proposition~\ref{prop:height basic}(i), \eqref{eq:C20 C21}, and \eqref{eq:C22 C23} imply
    \begin{equation}\label{eq:C24}
    (C_{21}/C_{23}^{1/N^2})^{e_1(n)+e_2(n)}<H(t_n)<C_{24}^{e_1(n)+e_2(n)}
    \end{equation}
    for all sufficiently large $n$ (by taking, for example, $C_{24}=C_{20}C_{23}$). Hence $H(t_n)\to\infty$ as $n\to\infty$.

    We now let $K=\Q(\theta_1,\theta_2,\gamma_1,\gamma_2)$. Let $S$ be the finite subset of $M_K$
    containing $M_K^{\infty}$ and all the finite places $v\in M_K^0$ such that
    either $\vert \theta_1\vert_v\neq 1$ or $\vert \theta_2\vert_v\neq 1$; in other words, $\theta_1$ and $\theta_2$ are $S$-units. For each large $n$, we define
    $$S_{n,>}=\{v\in S:\ \vert t_n\vert_v>1\}\ \text{and}\ S_{n,\leq}=\{v\in S:\ \vert t_n\vert_v\leq 1\}.$$
    Since there are only finitely many possibilities for the pair $(S_{n,>},S_{n,\leq})$
    of subsets of $S$, some pair must occur infinitely often. This means there exist an infinite set
    $\scrI$ of positive integers and 
    subsets
    $S_{>}$ and $S_{\leq}$ of $S$ such that for every $n\in\scrI$, we have
    $$S_{>}=\{v\in S:\ \vert t_n\vert_v>1\}\ \text{and}\ S_{\leq}=\{v\in S:\ \vert t_n\vert_v\leq 1\}.$$

    Let $\tilde{v}$ denote the restriction of the usual $\vert\cdot\vert$ from $\C$
    to the (real) field $K$. From \eqref{eq:approximation with C19}, we have:
    $$C_{25}<\vert t_n\vert<C_{26}$$
    for all sufficiently large $n$. Therefore
    \begin{equation}\label{eq:C27 C28}
        C_{27}<\vert t_n\vert_{\tilde{v}}=\vert t_n\vert^{1/[K:\Q]}<C_{28}
    \end{equation}
    We now define the family of elements $(t_v)_{v\in S}$ as follows:
    \begin{equation}
    t_v = \begin{cases}
  \gamma_2^{-D/c_2}/\gamma_1^{-D/c_1} & \text{if } v=\tilde{v} \\
  0 & \text{if } v\in S_{\leq}\setminus\{\tilde{v}\}\\
  \text{the symbol $\infty$} & \text{if } v\in S_{>}\setminus\{\tilde{v}\}
\end{cases}
    \end{equation}

    Let $n\in\scrI$, we have
    \begin{equation}\label{eq:LHS Roth}
        \prod_{v\in S}\min\{1,\vert t_n-t_v\vert_v\}<C_{19}^{-(N-2)(k_n+\ell_n)}\prod_{v\in S_{\leq}\setminus\{\tilde{v}\}} \vert t_n\vert_v\prod_{v\in S_{>}\setminus\{\tilde{v}\}}\frac{1}{\vert t_n\vert_v}
    \end{equation}
    thanks to \eqref{eq:approximation with C19} and our definition of $(t_v)_{v\in S}$.

    Since $\theta_1$ and $\theta_2$ are $S$-units, for all sufficiently large $n\in\scrI$, we have
    \begin{align}
        \begin{split}
            H(t_n)&=\prod_{v\in M_K}\max\{1,\vert t_n\vert_v\}\\
            &=\max\{1,\vert t_n\vert_{\tilde{v}}\}\cdot\prod_{v\notin S}\max\{1,\vert\alpha_n^{D/c_1}/\beta_n^{D/c_2}\vert_v\}\cdot\prod_{v\in S_{>}\setminus\{\tilde{v}\}}\vert t_n\vert_v\\
            &<C_{28}H(\alpha_n^{D/c_1}/\beta_n^{D/c_2})\prod_{v\in S_{>}\setminus\{\tilde{v}\}}\vert t_n\vert_v
        \end{split}
    \end{align}
    thanks to \eqref{eq:C27 C28}. Combining this with
    \eqref{eq:C22 C23}, we have
    \begin{equation}\label{eq:S> terms}
        \prod_{v\in S_{>}\setminus\{\tilde{v}\}}\frac{1}{\vert t_n\vert_v}<C_{28}C_{23}^{k_n+\ell_n}H(t_n)^{-1}
    \end{equation}
    Similarly, for all sufficiently large $n\in\scrI$, we have
    \begin{align}
        \begin{split}
            H(t_n)&=H(1/t_n)=\prod_{v\in M_K}\max\{1,1/\vert t_n\vert_v\}\\
            &<C_{29}H(\beta_n^{D/c_2}/\alpha_n^{D/c_1})\prod_{v\in S_{\leq}\setminus\{\tilde{v}\}}\frac{1}{\vert t_n\vert_v}
        \end{split}
    \end{align}
    Combining this with \eqref{eq:C22 C23}, we have
    \begin{equation}\label{eq:S< terms}
        \prod_{v\in S_{\leq}\setminus\{\tilde{v}\}} \vert t_n\vert_v<C_{29}C_{23}^{k_n+\ell_n}H(t_n)^{-1}
    \end{equation}

    We are now ready to specify $N$. At the beginning, we fix an integer $N>2$
    satisfying \eqref{eq:N condition epsilon1/c1}, \eqref{eq:N condition epsilon2/c2}, \eqref{eq:N condition C23 C21}, and the following inequality
    \begin{equation}\label{eq:N condition C19 C23}
        C_{19}^{N-2}>C_{23}^2.
    \end{equation}
    Then \eqref{eq:LHS Roth}, \eqref{eq:S> terms}, \eqref{eq:S< terms}, and \eqref{eq:N condition C19 C23} give
    \begin{align}\label{eq:LHS Roth simplified}
    \begin{split}
        \prod_{v\in S}\min\{1,\vert t_n-t_v\vert_v\}&<C_{23}^{-(k_n+\ell_n)}H(t_n)^{-2}\\
        &<C_{23}^{-(e_1(n)+e_2(n))/(2N^{C_{12}})}H(t_n)^{-2}
    \end{split}
    \end{align}
    for every sufficiently large $n\in \scrI$, where the second inequality follows from the lower bounds for $k_n$ and $\ell_n$ in
    \eqref{eq:from Section 3}. Define $\epsilon>0$ by:
    $$C_{23}^{1/(2N^{C_{12}})}=C_{24}^{\epsilon}$$
    Then \eqref{eq:C24} and \eqref{eq:LHS Roth simplified} yield
    $$\prod_{v\in S}\min\{1,\vert t_n-t_v\vert_v\}<H(t_n)^{-2-\epsilon}$$
    for all sufficiently large $n\in\scrI$. Then Theorem~\ref{thm:Roth} gives that
    $\{t_n:\ n\in\scrI\}$ is a finite set. This contradicts the property that $H(t_n)\to\infty$ as $n\to\infty$ and we finish the proof.

	\bibliographystyle{amsalpha}
	\bibliography{Salem3} 	
\end{document}